\documentclass[12pt]{amsart}

\usepackage{amssymb, latexsym}
\usepackage[dvipsnames]{color}

\theoremstyle{plain}
\newtheorem{theorem}{Theorem}[section]

\newtheorem{proposition}[theorem]{Proposition}
\newtheorem{lemma}[theorem]{Lemma}

\theoremstyle{definition}
\newtheorem{definition}[theorem]{Definition}
\newtheorem{remark}[theorem]{Remark}
\newtheorem{example}[theorem]{Example}

\numberwithin{equation}{section}

\usepackage[hypertex]{hyperref}

\input xy
\xyoption{all}

\begin{document}

\title[Isomorphism Invariants for Linear Quasigroups]{Isomorphism Invariants for Linear Quasigroups}
\author[J.D.H.~Smith]{Jonathan D.H.~Smith$^1$}
\author[S.G.~Wang]{Stefanie G.~Wang$^2$}
\address{ Department of Mathematics\\
Iowa State University\\
Ames, Iowa 50011, U.S.A.}
\email{$^1$jdhsmith@iastate.edu\phantom{,}}
\email{$^2$sgwang@iastate.edu}



\begin{abstract}
For a unital ring $S$, an $S$-linear quasigroup is a unital $S$-module, with automorphisms $\rho$ and $\lambda$ giving a (nonassociative) multiplication $x\cdot y=x^\rho+y^\lambda$. If $S$ is the field of complex numbers, then ordinary characters provide a complete linear isomorphism invariant for finite-dimensional $S$-linear quasigroups. Over other rings, it is an open problem to determine tractably computable isomorphism invariants. The paper investigates this isomorphism problem for $\mathbb{Z}$-linear quasigroups. We consider the extent to which ordinary characters classify $\mathbb{Z}$-linear quasigroups and their representations of the free group on two generators. We exhibit non-isomorphic $\mathbb{Z}$-linear quasigroups with the same ordinary character. For a subclass of $\mathbb{Z}$-linear quasigroups, equivalences of the corresponding ordinary representations are realized by permutational intertwinings. This leads to a new equivalence relation on $\mathbb{Z}$-linear quasigroups, namely permutational similarity. Like the earlier concept of central isotopy, permutational similarity is intermediate between isomorphism and isotopy.
\end{abstract}

\maketitle


\section{Introduction}\label{S:intro}

Quasigroups $(Q, \cdot)$ are nonassociative analogues of groups, retaining the cancelativity of the multiplication.  A pique\footnote{An acronym for ``Pointed Idempotent QUasigroup(E)".} $(P, \cdot, e)$ is a quasigroup with a nullary operation that selects an idempotent element $e$. The inner multiplication group of a pique $P$ is the stabilizer of $e$ in the group of permutations of $P$ generated by all the right and left multiplications.

For a commutative, unital ring $S$, an $S$-linear pique or $S$-linear quasigroup is an $S$-module $A$ equipped with automorphisms $\rho$ and $\lambda$ that furnish a pique multiplication $x\cdot y=x^\rho+y^\lambda$, with $0$ as the pointed idempotent. Two $S$-linear piques are $S$-isomorphic if they are isomorphic via an invertible $S$-linear transformation (module isomorphism).

Finite-dimensional $\mathbb{C}$-linear piques are classified up to $\mathbb{C}$-linear isomorphism by their so-called ordinary characters, obtained from the representation of the free group on two generators that they afford. Given a finite $\mathbb{Z}$-linear pique $A$, one may linearize the underlying combinatorial structure to obtain a $\mathbb{C}$-linear pique $\mathbb{C}A$, the so-called complexification of the $\mathbb{Z}$-linear pique $A$. Our primary concern is the extent to which ordinary characters of complexifications classify $\mathbb{Z}$-linear piques. The main result (Theorem~\ref{genfc}) shows that for a large class of $\mathbb{Z}$-linear pique structures, namely on cyclic groups of order not divisible by $8$, two piques that have the same complexified ordinary character are permutationally similar, i.e., the permutation actions of their respective inner multiplication groups are similar.

\subsection{Outline of the paper}\label{S:outline}

We begin with definitions and examples of quasigroups and linear quasigroups in Section \ref{S:LinQ}. We define linear quasigroups and their $S$-linear representations for a commutative unital ring $S$. Theorem \ref{isopique} identifies $S$-linear piques with the $S$-linear representations of the free group on two generators that they afford. This allows us to study the representations in lieu of the piques. Permutational similarity is defined in \S\ref{SS:permsmly}, while \S\ref{SS:ordichar} defines ordinary characters and the complexifications of $\mathbb{Z}$-linear piques. Theorem~\ref{T:issamech} observes that isomorphic $\mathbb{Z}$-linear piques have the same ordinary character.

The central Section~\ref{S:Zlinchar} considers isomorphism invariants for $\mathbb{Z}$-linear piques on cyclic groups of finite order not divisible by $8$. Linear piques defined on $\mathbb{Z}/_n$ for $n<5$ are classified up to isomorphism by the ordinary characters of their complexifications, the permutation characters introduced in Definition~\ref{D:char} (\S\ref{SS:linless5}). However, ordinary character theory does not suffice to cover all linear piques. Indeed, Proposition~\ref{P:Z5pspcni} exhibits non-isomorphic pique structures on $\mathbb{Z}/_5$ having the same permutation character. The main result (Theorem~\ref{genfc}) states that linear piques defined on cyclic groups of order not divisible by $8$, having the same permutation character, are permutationally similar.

The concluding Section~\ref{S:power2piques} examines the classification of $\mathbb{Z}$-linear pique structures on $\mathbb{Z}/_{2^n}$ for $n\geq 3$.  We exhibit pique structures on the cyclic group of order $16$, with the same permutation character, which are neither isomorphic nor permutationally similar (Theorem~\ref{T:16epcnps}).

\subsection{Related invariants}

Given a commutative, unital ring $S$ and an $S$-module $A$, the $S$-linear piques constructed on $A$ are all isotopic to the abelian group $(A,+,0)$. As such, their classification up to isomorphism may be regarded as a special case of the main problem considered by Dr\'{a}pal in \cite{GIHA}, namely the isomorphism problem for isotopes of a given (not necessarily abelian) group. However, the general solution to the isomorphism problem offered by \cite{GIHA} is computationally intractable, leaving open the search for less powerful but more accessible invariants. This situation is analogous to that prevailing in knot theory, where the existence of complete invariants does not preclude the continuing search for weaker invariants with a lower computational complexity.

\subsection{Conventions}

The paper follows the general algebraic convention of placing a function to the right of its argument, either on the line or as a superfix. This convention allows composites of functions to be read in natural order from left to right, and serves to minimize the occurrence of brackets, which otherwise proliferate when one studies non-associative structures.

\section{Linear piques}\label{S:LinQ}

\subsection{Quasigroups and piques}

\begin{definition}
A \emph{quasigroup} $(Q, \cdot,\backslash, /)$ is an algebra with three binary operations, multiplication $\cdot$, left division $\backslash$, and right division $/$, such that for all $x, y\in Q$,
\begin{align}
y\backslash(y\cdot x)&=x=(x\cdot y)/y\\
y\cdot (y\backslash x)&=x=(x/y)\cdot y
\end{align}
 are satisfied.
\end{definition}

\begin{definition} A \emph{pique} $(Q,\cdot,/,\backslash, e)$ is a quasigroup with a pointed idempotent element $e$ such that $e\cdot e=e$.
\end{definition}

\begin{definition}\cite[\S2.4]{IQTR}
Let $(Q,\cdot,/,\backslash, e)$ be a pique. The stabilizer of $e$ in the group of permutations of $Q$ generated by all the right multiplications $R(q)\colon x\mapsto x\cdot q$ and left multiplications $L(q)\colon x\mapsto q\cdot x$ (for $q\in Q$) is called the \emph{inner multiplication group} of the pique.
\end{definition}

A pique is a pointed set, where the idempotent element serves as the basepoint.  Maps between pointed sets send basepoint to basepoint. For pique homomorphisms, the pointed idempotent element of the domain maps to the pointed idempotent element of the codomain.

\begin{example}Each group is an associative pique, with the identity element as the pointed idempotent element. The inner multiplication group is the inner automorphism group.
\end{example}

\begin{example} Integers under subtraction form a nonassociative pique, with $0$ as the pointed idempotent. The unique nontrivial element of the inner multiplication group is negation.
\end{example}

\subsection{Linear piques}

\begin{definition}
Suppose that $S$ is a commutative, unital ring. A pique $(A,\cdot, /,\backslash,0)$ is said to be $S$-\emph{linear} if there is a unital $S$-module structure $(A,+,0)$, with automorphisms $\lambda$ and $\rho$ such that
\begin{align}
x\cdot y=x^\rho+y^\lambda, ~x/y=(x-y^\lambda)^{\rho^{-1}}, \text{and}~ x\backslash y=(y-x^\rho)^{\lambda^{-1}}
\end{align}
for $x, y\in A$.
\end{definition}

We identify $\lambda, \rho$ as the left and right multiplications by the pointed idempotent $0$.


\begin{example}
On the one hand, the quasigroup $(\mathbb{Z}/_4, x\circ_1 y)$ with the nonassociative multiplication $x\circ_1y=x(1\;2\;3)+y(1\;2)$ is a pique with 0 as the pointed idempotent element. However, neither (1\ 2\ 3) nor (1\ 2) is an automorphism of $\mathbb{Z}/_4$. On the other hand, the quasigroup $(\mathbb{Z}/_4, x\circ_2 y)$ with the nonassociative multiplication $x\circ_2y=x(1\ 3)+y(1\ 3)$ is also a pique with 0 as the pointed idempotent element. More importantly, the permutation $(1\;3)$ corresponds to the automorphism of $\mathbb{Z}/_4$ defined by $x\mapsto 3x$, so $(\mathbb{Z}/_4, x\circ_2 y)$ is a $\mathbb{Z}$-linear pique.
\end{example}

\begin{example}\cite[\S3]{Sm153}
Linear representations of two-generated groups are equivalent to piques.
\end{example}

\begin{remark}
Linear piques, and their shifted versions $x\cdot y=x^\rho+y^\lambda+c$ (also described as ``T-quasigroups" \cite{GBAT, NK}), have been studied for potential applications in cryptography and related fields \cite{DOSA}.
\end{remark}


\subsection{Equivalent representations}

Throughout this section, $S$ will denote a commutative, unital ring.

\begin{definition}\label{D:afforequ}
Let $\langle R, L\rangle$ be the free group on the doubleton $\{R,L\}$.
\begin{enumerate}
\item[$(\mathrm{a})$]
Let $(A, \cdot,/,\backslash, 0)$ be an $S$-linear pique with
$
x\cdot y=x^\rho+y^\lambda
$.
Then the group homomorphism
$$
\alpha\colon\langle R, L\rangle\to\mathrm{Aut}_S(A, +, 0);
R\mapsto\rho\,,\ L\mapsto\lambda
$$
is described as the \emph{$S$-linear representation} that is \emph{afforded} by $(A,\cdot, /,\backslash, 0)$.
\item[$(\mathrm{b})$]
Consider two $S$-modules $(A,+,0)$ and $(B,+,0)$. Then corresponding $S$-linear representations $\alpha\colon\langle R, L\rangle\to\mathrm{Aut}_S(A,+,0)$ and $\beta\colon\langle R, L\rangle\to\mathrm{Aut}_S(B,+,0)$ are
\emph{equivalent} whenever there exists an $S$-module isomorphism $f: A\to B$ such that for all $a$ in $A$ and $g$ in $\langle R, L\rangle$, the diagram
\begin{equation}\label{E:intrtwin}
\xymatrix{
A \ar[r]^{g^\alpha} \ar[d]_{f}
&
A \ar[d]^{f}
\\
B \ar[r]_{g^\beta}
&
B
}
\end{equation}
commutes. We call $f$ the \emph{intertwining}.
\end{enumerate}
\end{definition}

Note that the pair of equations
\begin{equation}\label{E:equivequ}
R^\alpha f=fR^\beta
\qquad
\mbox{ and }
\qquad
L^\alpha f=fL^\beta
\end{equation}
is equivalent to the commuting of \eqref{E:intrtwin}. Alternatively, one may require that the diagram
\begin{equation}\label{E:intrtwLR}
\xymatrix{
A \ar[d]_{f}
&
A \ar[r]^{R^\alpha} \ar[d]_{f}\ar[l]_{L^\alpha}
&
A \ar[d]^{f}
\\
B
&
B \ar[r]_{R^\beta} \ar[l]^{L^\beta}
&
B
}
\end{equation}
commutes.

\begin{lemma}\label{L:pisomequ}
Suppose that $f\colon(A, \circ_1)\to(B,\circ_2)$ is a pique isomorphism between $S$-linear piques $(A, \circ_1)$ and $(B,\circ_2)$. Let $\alpha$ and $\beta$ be the respective $S$-linear representations that they afford. Then the equations \eqref{E:equivequ} hold.
\end{lemma}

\begin{proof}
One has
\begin{align*}
aR^{\alpha}f
&=(a\circ_10)^f=a^f\circ_20=a^f R^{\beta}\
\mbox{ and}
\\
aL^{\alpha}f
&=(0\circ_1 a)^f=0\circ_2 a^f=a^f L^{\beta}.
\end{align*}
for each element $a$ of $A$.
\end{proof}

\begin{theorem}\label{isopique}
Let $(A, \circ_1)$ and $(B, \circ_2)$ be two $S$-linear piques. Then they are isomorphic by an $S$-linear transformation $f\colon A\to B$ if and only if the $S$-linear representations they afford are equivalent.
\end{theorem}

\begin{proof} Let $f\colon(A, \circ_1)\to(B,\circ_2)$ be an $S$-linear pique isomorphism. Suppose that $\alpha\colon\langle R, L \rangle\to\mathrm{Aut}(A, +, 0)$ and $\beta\colon\langle R, L\rangle\to\mathrm{Aut}(B,+,0)$ are the respective $S$-linear representations afforded by the $S$-linear piques. By Lemma~\ref{L:pisomequ}, the equations \eqref{E:equivequ} hold. It follows that $f$ is an intertwining witnessing the equivalence of $\alpha$ and $\beta$.

Now let $\alpha\colon \langle R, L\rangle\to\mathrm{Aut}(A,+,0)$ and $\beta\colon\langle R, L\rangle\to\mathrm{Aut}(B,+,0)$ be equivalent $S$-linear representations, with an intertwining $f\colon A\to B$. Then for $x, y$ in $A$, one has
\begin{align*}
	(x\circ_1y)f&=(xR^{\alpha}+yL^{\alpha})f\\
				&=(xR^{\alpha})f+(yL^{\alpha})f\\
				&=xf R^{\beta}+yf L^{\beta}\\
				&=xf\circ_2 yf\, ,
\end{align*}
so that $f\colon(A, \circ_1)\to(B, \circ_2)$ is an $S$-linear pique isomorphism.
\end{proof}

\subsection{Permutational similarity}\label{SS:permsmly}

In what follows, we will consider a modified version of the commuting diagram \eqref{E:intrtwLR}.

\begin{definition}\label{D:permsmly}
Let $A$ be a finite abelian group, with $\mathbb{Z}$-linear pique structures $(A,\circ_1)$ and $(A,\circ_2)$ affording respective representations
$$
\alpha_i\colon\langle R,L\rangle\to\mathrm{Aut}(A,+,0)
$$
for $i=1,2$. Then the piques $(A,\circ_1)$ and $(A,\circ_2)$, or the representations they afford, are said to be \emph{permutationally similar}, via a permutation $\pi$ of the underlying set $A$, if the diagram
\begin{equation}\label{E:persimdg}
\xymatrix{
A \ar[d]_{\pi}
&
A \ar[r]^{R^{\alpha_1}} \ar[d]_{\pi}\ar[l]_{L^{\alpha_1}}
&
A \ar[d]^{\pi}
\\
A
&
A \ar[r]_{R^{\alpha_2}} \ar[l]^{L^{\alpha_2}}
&
A
}
\end{equation}
commutes. In other words, the permutation $\pi$ conjugates both $R^{\alpha_1}$ to $R^{\alpha_2}$ and $L^{\alpha_1}$ to $L^{\alpha_2}$ within the permutation group $A!$ of the set $A$.
\end{definition}

Consider two permutationally similar piques $(A,\circ_1)$ and $(A,\circ_2)$ as in Definition~\ref{D:permsmly}. If $\pi$ is not an automorphism of the abelian group $(A,+,0)$, then the permutational similarity of representations furnished by $\pi$ is not an equivalence in the sense of Definition~\ref{D:afforequ}. On the other hand, since both the piques are isotopic to the abelian group $A$, they are mutually isotopic. Furthermore, Theorem~\ref{isopique} shows that if two $\mathbb{Z}$-linear piques on the abelian group $A$ are isomorphic, then they are permutationally similar. Thus permutational similarity is a relationship intermediate between isotopy and isomorphism. As such, it is analogous to the relationship of central isotopy \cite[\S3.4]{IQTR}.

\subsection{Ordinary characters of $\mathbb{C}$-linear piques}\label{SS:ordichar}


\begin{definition}
Let $G$ be a group. For a complex vector space $V$, let $\mathsf{GL}(V)$ be its group of automorphisms.
\begin{enumerate}
\item[$(\mathrm{a})$]
An \emph{ordinary linear representation} of $G$ is defined as a homomorphism $\rho: G\to \mathsf{GL}(V)$, for some finite-dimensional complex vector space $V$.
\item[$(\mathrm{b})$]
The (\emph{ordinary}) \emph{character} of an ordinary linear representation $\rho\colon G\to\mathsf{GL}(V)$ is the function $\chi$ or $\chi_\rho\colon G\to \mathbb{C}; g\mapsto \mathrm{Tr}(g\rho)$.
\end{enumerate}
\end{definition}

\begin{definition}\label{D:cmplxftn}
Let $(A, \cdot, /,\backslash, 0)$ be a finite $\mathbb{Z}$-linear pique, affording the $\mathbb{Z}$-linear representation $\alpha:\langle R, L\rangle\to\mathrm{Aut}(A,+,0)$. Let $\mathbb{C}A$ be the complex vector space with basis $A$. Then the \emph{complexification} of $(A, \cdot, /,\backslash, 0)$ is the $\mathbb{C}$-linear pique structure $(\mathbb{C}A, \cdot, /,\backslash, 0)$ obtained by extension of the pique structure $(A, \cdot, /,\backslash, 0)$. Thus
$$
\alpha_\mathbb{C}\colon
R\mapsto(\mathbb{C}A\to\mathbb{C}A;a\mapsto aR^\alpha)\,,\
L\mapsto(\mathbb{C}A\to\mathbb{C}A;a\mapsto aL^\alpha)
$$
serves to specify the $\mathbb{C}$-linear representation $\alpha_\mathbb{C}$ that is afforded by the complexification of $(A, \cdot, /,\backslash, 0)$.
\end{definition}

\begin{theorem}\label{T:issamech}
Let $f\colon(A, \cdot, /,\backslash, 0)\to(B, \cdot, /,\backslash, 0)$ be an isomorphism of finite $\mathbb{Z}$-linear piques affording respective $\mathbb{Z}$-linear representations $\alpha$ and $\beta$. Then the respective $\mathbb{C}$-linear representations $\alpha_\mathbb{C}$ and $\beta_\mathbb{C}$ of their complexifications have the same ordinary character.
\end{theorem}

\begin{proof}
The bijection $f\colon A\to B$ may be extended to a unique $\mathbb{C}$-linear isomorphism $\mathbb{C}f\colon\mathbb{C}A\to\mathbb{C}B$. By Lemma~\ref{L:pisomequ}, one has $g^{\alpha}f=fg^{\beta}$ for all $g$ in $\langle R,L\rangle$. By linearity, one then has $g^{\alpha_\mathbb{C}}(\mathbb{C}f)=(\mathbb{C}f)g^{\beta_\mathbb{C}}$ for all $g$ in $\langle R,L\rangle$. Let $\chi_A$ and $\chi_B$ be the respective characters of $\alpha_\mathbb{C}$ and $\beta_\mathbb{C}$. Then
\begin{align*}
\chi_B(g\beta_\mathbb{C})
&=\mathrm{Tr}(g{\beta_\mathbb{C}})
=\mathrm{Tr}\big((\mathbb{C}f)^{-1}g^{\alpha_\mathbb{C}}(\mathbb{C}f)\big)
=\mathrm{Tr}(g{\alpha_\mathbb{C}})
=\chi_A(g\alpha_\mathbb{C})
\end{align*}
for each $g$ in $\langle R,L\rangle$.
\end{proof}

\begin{remark}
Although the result will not be needed for subsequent work in the current paper, it should be noted that Theorem~\ref{isopique}, along with Theorem~12.4 of \cite{IQTR}, implies that finite-dimensional $\mathbb{C}$-linear quasigroups are classified, up to $\mathbb{C}$-linear isomorphism, by their ordinary characters.
\end{remark}

\section{Linear piques on finite cyclic groups} \label{S:Zlinchar}

\subsection{Permutation characters}

The following definition provides a purely combinatorial specification for the character of the ordinary representation that is afforded by the complexification of a finite $\mathbb{Z}$-linear pique (compare \cite[Exercise~2.2]{LRFG}).

\begin{definition}\label{D:char}
Let $(A, \cdot, /,\backslash, 0)$ be a finite $\mathbb{Z}$-linear pique, affording the $\mathbb{Z}$-linear representation $\alpha:\langle R, L\rangle\to\mathrm{Aut}(A,+,0)$. For an element $g$ of $\langle R, L\rangle$, the \emph{permutation character} $\chi(g)$ is the number of fixed points of the permutation $g^\alpha$ of the set $A$.
\end{definition}

Although the group $\langle R, L\rangle$ is infinite, the permutation character is determined by the fixed-point numbers for each member of the finite set $\langle R, L\rangle^\alpha$ of permutations of $A$, the inner multiplication group of the pique  $(A, \cdot, /,\backslash, 0)$. We generally use cycle notation for permutations of $A$, recognizing the number of fixed points of a permutation as the number of one-cycles in its cycle decomposition.

For $1<n\in\mathbb{Z}$, we will consider $\mathbb{Z}$-linear piques defined on finite cyclic groups $(\mathbb{Z}/_n,+.0)$. We write $(\mathbb{Z}/_n)^*$ for the group of units of the monoid $(\mathbb{Z}/_n,\cdot, 1)$, the set of residues coprime to $n$. We use the isomorphism
$$
(\mathbb{Z}/_n)^*\to\mathrm{Aut}(\mathbb{Z}/_n,+,0);
r\mapsto(x\mapsto rx)
$$
\cite[5.7.11]{Scott} to identify automorphisms of finite cyclic groups. Thus the order of the automorphism group $\mathrm{Aut}(\mathbb{Z}/_n,+,0)$ is given by the Euler function $\varphi(n)$. We note the following for future reference.

\begin{lemma}\label{fixedpoints}
Let $p$ be a prime number, and let $k$ be a positive integer. Then an automorphism of $\mathbb{Z}/_{p^k}$ has $p^j$ many fixed points, for some $0<j\leq k$.
\end{lemma}

\begin{proof}
The set of fixed points of a group automorphism forms a subgroup of the group in question. The result then follows by Lagrange's Theorem.
\end{proof}

\subsection{Linear piques on small cyclic groups}\label{SS:linless5}

We build piques on $\mathbb{Z}/_3$ by assigning automorphisms of $(\mathbb{Z}/_3, +,0)$ to $R, L$. Since $R, L$ can be 1 or 2, we have four possibilities for the binary multiplication. Here, we exhibit the permutation character table for $\mathbb{Z}$-linear representations of each linear pique defined on $\mathbb{Z}/_3$.

\begin{table}[htb]
\begin{center}
\begin{tabular}{|c||c|c|c|c|c|}
\hline
$x \cdot y$ 	& $R$ 	& $L$ 	& $\chi(R)$& $\chi(L)$	\\ \hline\hline
$x+y$			&	(1)		&	(1)	 	& 3		& 3			\\ \hline
$x+2y$		&	(1)		&	(1\ 2) 	& 3		& 1			\\ \hline
 $2x+y$		&	(1\ 2)	&	(1)		& 1		& 3			\\ \hline
 $2x+2y$		&	(1\ 2)	&	(1\ 2) 	& 1		& 1			\\ \hline
\end{tabular}
\end{center}
\vskip 2mm
\caption{Permutation characters for linear piques on $\mathbb{Z}/_3$}\label{z3char}
\end{table}

The ordinary characters of $R, L$ are distinct for linear piques of order $3$. By Theorem~\ref{T:issamech}, the four piques are all mutually non-isomorphic. Thus the permutation character completely resolves the isomorphism classes of linear piques of order $3$:

\begin{proposition}
Linear piques defined on $\mathbb{Z}/_3$ are classified completely up to isomorphism by their permutation characters.
\end{proposition}

In similar vein, one obtains the following:

\begin{proposition}\label{P:z4char}
Linear piques defined on each of $\mathbb{Z}/_2$ and $\mathbb{Z}/_4$ are classified completely up to isomorphism by their permutation characters.
\end{proposition}

\subsection{Cyclic groups of prime power order}

Consider a cyclic group $\mathbb{Z}/_{p^k}$, where $p$ is a prime and $k$ is a positive integer.

\begin{lemma}\cite[5.7.12]{Scott}
If $p$ is an odd prime and $k$ is a positive integer, or $p=2$ and $k\in\{1,2\}$, then $\mathrm{Aut}(\mathbb{Z}/_{p^k},+,0)$ is a cyclic group of order $\varphi(p^k)=p^{k-1}(p-1)$.
\end{lemma}

\begin{lemma}\label{fruitcakelemma1}
Let $\mathbb{Z}$-linear representations $\alpha_i:\langle R, L\rangle\to\mathrm{Aut}(\mathbb{Z}/_{p^k})$ have equal respective permutation characters $\chi_i$, for $i=1,2$. Then for each element $g$ of $\langle R, L\rangle$, the automorphisms $g^{\alpha_1}$ and $g^{\alpha_2}$ have the same order.
\end{lemma}

\begin{proof}
Suppose, without loss of generality, that $s=|\langle g^{\alpha_1}\rangle|\ge|\langle g^{\alpha_2}\rangle|=t$. Then $\chi_1(g^t)=\chi_2(g^t)=p^k$, so that $g^{\alpha_1t}=g^{t\alpha_1}=1$ and $s\le t$.
\end{proof}

\begin{lemma}\cite[Ex.~2.1]{PCPG}\label{L:PCPGPCPG}
Two permutation representations of a finite cyclic group, with the same character, are isomorphic.
\end{lemma}

\begin{proposition}\label{fruitcake}
Let $p$ be a prime number, and let $k$ be a positive integer. Suppose that two linear pique structures defined on $\mathbb{Z}/p^k$ have the same permutation character. Suppose that one of the three following hypotheses applies:
\begin{enumerate}
\item[$(\mathrm{a})$]
Let $p$ be an odd prime;
\item[$(\mathrm{b})$]
Let $p=2$ and $k\in\{1,2\}$;
\item[$(\mathrm{c})$]
Let $p=2$ and $k>2$, but assume that the inner multiplication groups of the two piques are cyclic.
\end{enumerate}
Then the corresponding representations are permutationally similar.
\end{proposition}

\begin{proof}
Suppose that the piques correspond to respective representations $\alpha_i:\langle R, L\rangle\to\mathrm{Aut}(\mathbb{Z}/_{p^k})$, for $i=1,2$. Suppose, without loss of generality, that $|\langle R,L\rangle^{\alpha_1}|\ge|\langle R,L\rangle^{\alpha_2}|$. Let $g$ be an element of $\langle R, L\rangle$ whose image under $\alpha_1$ generates $\langle R,L\rangle^{\alpha_1}$, so the order of $g^{\alpha_1}$ is $|\langle R,L\rangle^{\alpha_1}|$. Then by Lemma~\ref{fruitcakelemma1}, the order of $g^{\alpha_2}$ is $|\langle R,L\rangle^{\alpha_1}|$. Thus $|\langle R,L\rangle^{\alpha_1}|=|\langle R,L\rangle^{\alpha_2}|$, and $g^{\alpha_2}$ generates $\langle R,L\rangle^{\alpha_2}$. Consider the finite cyclic group $G\cong\langle g^{\alpha_1}\rangle\cong\langle g^{\alpha_2}\rangle$, with permutation representations $\gamma_i\colon G\to\mathrm{Aut}(\mathbb{Z}/_{p^k});g^{\alpha_it}\mapsto g^{t\alpha_i}$ for $i=1,2$. The respective permutation characters are equal, so by Lemma~\ref{L:PCPGPCPG}, the two permutation representations $\gamma_i$ of $G$ are isomorphic. It follows that the representations $\alpha_1,\alpha_2$ are permutationally similar.
\end{proof}

\subsection{Cyclic groups of order not divisible by 8}

For any positive integer $m$, consider a factorization
$$
m=\prod_{i=1}^sp_i^{k_i}
$$
with distinct primes $p_1<\ldots<p_s$ for $1\le i\le s$. Write $q_i=p_i^{k_i}$ for $1\le i\le s$. We refer to $q_i$ as the $p_i$-\emph{part} of $m$. Now for $1<n\in\mathbb{Z}$, fix the notation $n=\prod_{i=1}^sp_i^{k_i}$, with distinct primes $p_1<\ldots<p_s$ and positive exponents $k_1,\dots,k_s$, for $1\le i\le s$.

\begin{proposition}\label{CRT}\cite[5.7.3]{Scott}
Let $A$ be an abelian group of order $n$. For $1\le i\le s$, let $A_i$ be the Sylow $p_i$-subgroup of $A$. Then $\mathrm{Aut}(A)\cong \prod_{i=1}^s\mathrm{Aut}(A_i)$.
\end{proposition}

The Chinese Remainder Theorem gives a direct sum decomposition
\begin{equation}\label{E:Chinese}
c\colon\mathbb{Z}/_n\to\bigoplus_{i=1}^s\mathbb{Z}/_{q_i};
x\mapsto(x_1,\dots,x_s) \,.
\end{equation}
In turn, application of Proposition~\ref{CRT} to the cyclic group $\mathbb{Z}/_n$ yields the isomorphism
\begin{equation}\label{E:Scottie}
a\colon\mathrm{Aut}(\mathbb{Z}/_n)\to\prod_{i=1}^s\mathrm{Aut}(\mathbb{Z}/_{q_i});
\theta\mapsto(\theta_1,\dots,\theta_s) \,.
\end{equation}
For an automorphism $\theta$ of $\mathbb{Z}/_n$, let $\pi(\theta)$ be the number of fixed points of $\theta$. For $1\le i\le s$, let $\pi_i(\theta_i)$ be the number of fixed points of $\theta_i$ on $\mathbb{Z}/_{q_i}$. By virtue of the set isomorphism $c\colon\mathbb{Z}/_n\to\prod_{i=1}^s\mathbb{Z}/_{q_i}$, one has
$$
\pi(\theta)=\prod_{i=1}^s\pi_i(\theta_i) \, .
$$
Then by Lemma~\ref{fixedpoints}, $\pi_i(\theta_i)$ is the $p_i$-part of $\pi(\theta)$.

Now restrict the fixed integer $n$ by requiring that it not be divisible by $8$. In our notation, this means that $k_1<3$ if $p_1=2$. As a consequence, the automorphism groups $\mathrm{Aut}(\mathbb{Z}/_{q_i})$ are all cyclic.

\begin{theorem}\label{genfc}
Let $A$ be a finite cyclic group whose order is not divisible by $8$. Then if two $\mathbb{Z}$-linear piques on $A$ have the same permutation character, they are permutationally similar.
\end{theorem}

\begin{proof}
By transport of structure, it suffices to examine the case where $A=\mathbb{Z}/_n$, with notation as above. Consider the representations $\alpha,\alpha'$ of $\langle R,L\rangle$ corresponding to the two pique structures. Suppose that their respective permutation characters are $\chi$ and $\chi'$. By the hypothesis, these characters coincide. In particular, for each element $g$ of $\langle R,L\rangle$, and for each $1\le i\le s$, the respective $p_i$-parts of $\chi_i(g)$ and $\chi_i'(g)$ of $\chi(g)$ and $\chi'(g)$ coincide.

For each $1\le i\le s$, and for each element $g$ of $\langle R,L\rangle$, define
$
g^{\alpha_i}=(g^{\alpha})_i
$
and
$
g^{\alpha_i'}=(g^{\alpha'})_i
$
using the notation embodied in \eqref{E:Scottie}. One obtains respective representations $\alpha_i$ and $\alpha_i'$ of $\langle R,L\rangle$ on $\mathbb{Z}/_{q_i}$, with equal permutation characters $\chi_i(g)$ and $\chi_i'(g)$. By Proposition~\ref{fruitcake}, it follows that these representations are permutationally similar, say by permutations $b_i\colon\mathbb{Z}/_{q_i}\to\mathbb{Z}/_{q_i}$. Then the permutation $b$ of $\mathbb{Z}/_n$, defined by setting $xb=(x_1b_1,\dots,x_sb_s)c^{-1}$ in the notation of \eqref{E:Chinese}, yields the desired permutation similarity between $\alpha$ and $\alpha'$.
\end{proof}

\subsection{Linear piques on $\mathbb{Z}/_5$}

Now we consider an explicit example of the preceding work using linear piques defined on $\mathbb{Z}/_5$. Automorphisms of $\mathbb{Z}/_5$ are given by multiplication by non-zero elements. The following table lists the permutations for each element in $(\mathbb{Z}/_5)^*$.

\begin{center}
\begin{tabular}{|c|c|c|c|c|}
\hline
Automorphism	& 1	& 2			& 3			&4	\\ \hline
Permutation	& (1)& (1\;2\;4\;3)	& (1\;3\;4\;2)	& (1\;4)(2\;3) \\ \hline
\end{tabular}
\end{center}

Let $x\circ_1 y=x+2y$ and $x\circ_2y=x+3y$. Since the identity $(xx)x=(yy)y$ holds in $(\mathbb{Z}/_5, \circ_1)$, but not in $(\mathbb{Z}/_5, \circ_2)$, the respective piques are certainly not isomorphic, even as magmas under the quasigroup multiplication.

On the other hand, the representations of $(\mathbb{Z}/_5, \circ_1)$ and $(\mathbb{Z}/_5, \circ_2)$ have the same permutation character. In each case, $R$ maps to the identity, and $L$ maps to a $4$-cycle. Let $\{e_i\mid 0\le i<5\}$ be the standard basis for $\mathbb{C}^5$. Consider the permutation matrix
\begin{equation*}
P_{(2\ 3)}=
\begin{bmatrix}
1 & 0 & 0 & 0 & 0 \\
0 & 1 & 0 & 0 & 0 \\
0 & 0 & 0 & 1 & 0 \\
0 & 0 & 1 & 0 & 0 \\
0 & 0 & 0 & 0 & 1
\end{bmatrix}
\end{equation*}
of the permutation $(2\ 3)$. Define the linear transformation
$$
\tau:\mathbb{C}^5\to \mathbb{C}^5; e_i\mapsto e_iP_{(2\;3)} \, .
$$
Since the 4-cycles (1\ 2\ 4\ 3) and (1\ 3\ 4\ 2) are conjugated by (2\ 3), the ordinary representations for the non-isomorphic $\mathbb{Z}$-linear piques $(\mathbb{Z}/_5,\circ_1)$ and $(\mathbb{Z}/_5,\circ_2)$ are permutationally similar. We may summarize as follows.

\begin{proposition}\label{P:Z5pspcni}
There is a pair of $\mathbb{Z}$-linear piques on $\mathbb{Z}/_{5}$ which have the same permutation character, and are permutationally similar, but which are not isomorphic.
\end{proposition}

\section{Linear piques defined on $\mathbb{Z}/_{2^k}$}\label{S:power2piques}

As recorded in Proposition~\ref{P:z4char}, linear piques defined on $\mathbb{Z}/_2$ and $\mathbb{Z}/_4$ are classified up to isomorphism by their permutation characters. In this section, we examine the classification of $\mathbb{Z}$-linear pique structures on $\mathbb{Z}/_{2^k}$ for $k\geq 3$. For each positive integer $k$,
the group of units of the monoid of integers modulo $2^k$ consists of the non-zero odd residues.

\subsection{The case of $\mathbb{Z}/_8$}
Let us consider linear piques defined on $\mathbb{Z}/_8$. To construct a $\mathbb{Z}$-linear pique on $\mathbb{Z}/_8$, we must assign $\rho,\lambda$ the values 1, 3, 5, or 7. The following table lists the permutations for each element in $(\mathbb{Z}/_8)^*$.

\begin{center}
\begin{tabular}{|c|c|c|c|c|}
\hline
Automorphism	& 1	& 3				& 5			&7	\\ \hline
Permutation	& (1)& (1\;3)(2\;6)(5\;7)& (1\;5)(3\;7)	& (1\;7)(2\;6)(3\;5) \\ \hline
\end{tabular}
\end{center}

If two linear piques have the same permutation characters, then the permutations associated with $R, L$ must have the same cycle type. The only possibilities for isomorphic ordinary representations are listed in the following table. We omit opposite quasigroups.

\begin{table}[hbt]
\begin{center}
\begin{tabular}{|c|c|c|c|c|c|c|}
\hline
$x\cdot y$	& $R$	& $L$		& $\chi(R)$& $\chi(L)$ & $\chi(L^2)$	 & $\chi(RL)$ \\ \hline
$x+3y$		& (1)			& (1\;3)(2\;6)(5\;7)	& 8		& 2		& 8			& 2			 \\ \hline
$x+7y$		&(1)			& (1\;7)(2\;6)(3\;5)	& 8		& 2		& 8			& 2			 \\ \hline
$5x+3y$		&(1\;5)(3\;7)	& (1\;3)(2\;6)(5\;7)	& 4		& 2		& 8			& 2			 \\ \hline
$5x+7y$		&(1\;5)(3\;7)	& (1\;7)(2\;6)(3\;5)	& 4		& 2		& 8			& 2			 \\ \hline
\end{tabular}
\end{center}
\vskip 2mm
\caption{Partial character table for linear piques on $\mathbb{Z}/_8$}\label{z8char}
\end{table}

The permutations for $3$ and $7$ are conjugated by $(3\ 7)$, while those for $1$ and $5$ are fixed under conjugation by $(3\ 7)$. Thus for equivalent complexified ordinary representations, the permutation matrix $P_{(3\;7)}$ serves as a permutation intertwining. We may summarize as follows.

\begin{proposition}\label{P:Z8pspcni}
If a pair of $\mathbb{Z}$-linear piques on $\mathbb{Z}/_{8}$ have the same permutation character, then they are permutationally similar.
\end{proposition}

\subsection{Computing permutations for automorphisms of $\mathbb{Z}/_{2^k}$}

For consideration of linear piques defined on $\mathbb{Z}/_{2^k}$ for $k\geq 4$, it becomes unwieldy to determine the permutations implemented by each automorphism  of $\mathbb{Z}/_{2^k}$ by hand. Instead, we use a program to list the permutations, and to enumerate their fixed points. We illustrate the process by computing the permutations for automorphisms of $\mathbb{Z}/_{16}$.

The residue $1$ corresponds to the identity permutation. To compute the permutation for the automorphism $x\mapsto 3x$, the program generates cycles in disjoint cycle notation as follows: $1*3=3,\ 3*3=9,\ 9*3=11,\ 11*3=1$.
Once an element is congruent to the starting element of the cycle (in this case $1$), the program stops the process and outputs the cycle --- here $(1\ 3\ 9\ 11)$. The elements that appear in this cycle are removed from the list of odd integers modulo 16. The program takes the next smallest element from the list of remaining integers modulo $16$, and repeats the process. Since $2*3=6$ and $6*3=2$ modulo $16$,
this computation gives us the transposition $(2\ 6)$. The program appends it to the first cycle, so we have $(1\ 3\ 9\ 11)(2\ 6)$. Then $2$ and $6$ are removed from the list, and the process continues. The program stops when the list of remaining integers modulo $16$ is empty. Once the process is complete for a given automorphism, it moves on to the next smallest representative element of $(\mathbb{Z}/_{16})^*$, until there are no more.

For each permutation, the program computes the number of fixed points. For a given element $u$ in $(\mathbb{Z}/_{16})^*$, the program checks if $au\equiv a\mod16$ for each $a\in\mathbb{Z}/_{16}$. If the equation holds, then the program adds $1$ to the number of fixed points for the permutation associated with $u$. The information is compiled into Table~\ref{z16fixedpts}.

\begin{table}[hbt]
\centering
\begin{tabular}{|c|c|c|}
\hline
Autom. & Permutation & Fixed points\\ \hline
1	& (1) & 16\\ \hline
3	& (1\;3\;9\;11)(2\;6)(4\;12)(5\;15\;13\;7)(10\;14)	& 2\\ \hline
5	& (1\;5\;9\;13)(2\;10)(3\;5\;11\;7)(6\;14)		& 4\\ \hline
7	& (1\;7)(2\;14)(3\;5)(4\;12)(6\;10)(9\;15)(11\;13)	& 2\\ \hline
9	& (1\;9)(3\;11)(5\;13)(7\;15)			& 8 \\ \hline
11	& (1\;11\;9\;3)(2\;6)(4\;12)(5\;7\;13\;15)(10\;14) & 2	\\ \hline
13	& (1\;13\;9\;5)(2\;10)(3\;7\;11\;15)(6\;14)		& 4 \\ \hline
15	& (1\;15)(2\;14)(3\;13)(4\;12)(5\;11)(6\;10)(7\;9)	& 2\\ \hline
\end{tabular}
\vskip 2mm
\caption{Automorphisms of $\mathbb{Z}/_{16}$: permutations and fixed point counts}\label{z16fixedpts}
\end{table}

\subsection{Linear piques on $\mathbb{Z}/_{16}$}

Linear piques on $\mathbb{Z}/_{16}$ are summarized in Table~\ref{z16char}. We are only concerned with piques having non-cyclic inner multiplication groups, since the piques with cyclic inner multiplication groups are handled by Proposition~\ref{fruitcake}(c). In addition, we have chosen single representatives from each pair of mutually opposite quasigroups. In the first column of the table, each pique is identified by the respective right multiplication $\rho$ and left multiplication $\lambda$ by $0$.


For each pique listed, the table provides summary information on the permutation character. Note that elements in $(\mathbb{Z}/_{16})^*$ have order $1$, $2$, or $4$, since $(\mathbb{Z}/_{16})^*\equiv C_2\times C_4$. Thus when considering which words from $\langle R, L\rangle$ will have their permutation character displayed in Table~\ref{z16char}, it suffices to take powers of $R, L$ strictly less than $4$.


\begin{table}[hbt]
\centering
\begin{tabular}{|c|c|c|c|c|c|c|c|c|}
\hline
$\rho$ & $\lambda$ & $\chi(RL)$ &$\chi(RL^2)$ & $\chi(RL^3)$ & $\chi(R)$ & $\chi(R^2)$ & $\chi(R^3)$ &$\chi(L^2)$ \\ \hline\hline
	5 & 3 & 2 & 4 & 2 & 4 & 8 & 4 & 8 \\ \hline
	5 & 11 & 2 & 4 & 2 & 4 & 8 & 4 & 8 \\ \hline
	13 & 3  & 2 & 4 & 2 & 4 & 8 & 4 & 8 \\ \hline
	13 & 11 & 2 & 4 & 2 & 4 & 8 & 4 & 8 \\ \hline\hline
	13 & 7 & 2 & 4 & 2 & 4 & 8 & 4 & 16 \\ \hline
	13 & 15  & 2 & 4 & 2 & 4 & 8 & 4 & 16 \\ \hline
	5 & 7 & 2 & 4 & 2 & 4 & 8 & 4 & 16 \\ \hline
	5 & 15 & 2 & 4 & 2 & 4 & 8 & 4 & 16 \\ \hline\hline
	9 & 7 & 2 & 8 & 2 & 8 & 16 & 8 & 16 \\ \hline
	9 & 15& 2 & 8 & 2 & 8 & 16 & 8 & 16 \\ \hline\hline
	11 & 7 & 4 & 2 & 4 & 2 & 8 & 2 & 16 \\ \hline
	3 & 15 & 4 & 2 & 4 & 2 & 8 & 2 & 16 \\ \hline
	3 & 7 & 4 & 2 & 4 & 2 & 8 & 2 & 16 \\ \hline
	11 & 15 & 4 & 2 & 4 & 2 & 8 & 2 & 16 \\ \hline\hline
	7 & 15 & 8 & 2 & 8 & 2 & 16 & 2 & 16 \\ \hline
\end{tabular}
\vskip 2mm
\caption{Partial character table for linear piques on $\mathbb{Z}/_{16}$}\label{z16char}
\end{table}





The top four rows of the body of Table~\ref{z16char} exhibit four linear piques that yield the same permutation character. Since we want to consider a pair of piques, we have $6$ options. Take the permutations $\sigma=(3\ 11)(7\ 15)$ and $\tau=(7\ 15)(5\ 13)$. If the two piques in the pair have identical $\rho$ or identical $\lambda$, they will yield permutations that can be simultaneously conjugated by $\sigma$ or $\tau$. Thus four pairings of the linear piques in the top four rows of the body of Table~\ref{z16char} have representations that are permutationally similar.

Now consider the two linear piques $(\mathbb{Z}/_{16}, \circ_1)$ with $x\circ_1y=5x+3y$ and $(\mathbb{Z}/_{16}, \circ_2)$ with $x\circ_2y=13x+11y$. The right and left multiplications cannot be simultaneously conjugated by $\sigma$ or $\tau$. In the conjugation of $3$ and $11$, $5$ and $13$ are fixed points. In the conjugation of $5$ and $13$, $3$ and $11$ now become fixed points. If there exists $\pi \in S_{16}$ that simultaneously conjugates these pairs of permutations, $\pi$ needs both to fix and to interchange $3$ and $11$, $5$ and $13$ in the respective permutations. This is impossible. Hence, the piques $(\mathbb{Z}/_{16}, \circ_1)$ and $(\mathbb{Z}/_{16}, \circ_2)$ are not permutationally similar. The pair of linear piques with multiplications given by $5x+11y$ and $13x+3y$ displays the same behavior. Summarizing, we have obtained the following negative result to contrast with the positive results obtained earlier, along with Proposition~\ref{P:Z5pspcni}.

\begin{theorem}\label{T:16epcnps}
There are pairs of $\mathbb{Z}$-linear piques on $\mathbb{Z}/_{16}$ which have the same permutation character, but which are neither isomorphic nor permutationally similar.
\end{theorem}

\end{document}